\documentclass[a4paper, 11pt]{amsart} 
\usepackage{amsmath,amsthm,amsfonts,amssymb,amscd} 
\usepackage[all]{xy}
\usepackage{enumerate}
\usepackage{here}
\usepackage{float}
\usepackage{booktabs}
\usepackage{bookmark}

\newtheorem{thm}{Theorem}[section] 
\newtheorem{prop}[thm]{Proposition} 
\newtheorem{lem}[thm]{Lemma} 
\newtheorem{cor}[thm]{Corollary} 
\newtheorem{fact}[thm]{Fact} 
\theoremstyle{definition} 
\newtheorem{rem}[thm]{Remark} 

\newcommand{\N}{\mathbb{N}}

\newcommand{\R}{\mathbb{R}}
\newcommand{\C}{\mathbb{C}}
\renewcommand{\H}{\mathbb{H}}

\renewcommand{\a}{\mathfrak{a}}
\renewcommand{\c}{\mathfrak{c}}
\newcommand{\e}{\mathfrak{e}}
\newcommand{\f}{\mathfrak{f}}
\newcommand{\g}{\mathfrak{g}}
\newcommand{\h}{\mathfrak{h}}
\renewcommand{\j}{\mathfrak{j}}
\renewcommand{\k}{\mathfrak{k}}
\newcommand{\mfl}{\mathfrak{l}}

\newcommand{\p}{\mathfrak{p}}
\newcommand{\q}{\mathfrak{q}}
\renewcommand{\t}{\mathfrak{t}}
\newcommand{\z}{\mathfrak{z}}

\newcommand{\gl}{\mathfrak{gl}}
\renewcommand{\sl}{\mathfrak{sl}}
\newcommand{\su}{\mathfrak{su}}
\newcommand{\so}{\mathfrak{so}}
\renewcommand{\sp}{\mathfrak{sp}}

\newcommand{\SL}{\operatorname{SL}}
\newcommand{\SU}{\operatorname{SU}}
\newcommand{\SO}{\operatorname{SO}}
\newcommand{\Sp}{\operatorname{Sp}}

\newcommand{\V}{\mathcal{V}}

\newcommand{\bl}{\bullet}
\newcommand{\bs}{\backslash}
\newcommand{\ep}{\varepsilon}
\newcommand{\halfep}{\delta}
\newcommand{\simto}{\xrightarrow{\sim}}

\newcommand{\im}{\operatorname{im}}

\newcommand{\rest}{\operatorname{rest}}
\newcommand{\rank}{\operatorname{rank}}
\newcommand{\sgn}{\operatorname{sgn}}
\newcommand{\Ad}{\operatorname{Ad}}
\newcommand{\cdR}{\operatorname{cd}_\R}
\newcommand{\Lie}{\operatorname{Lie}}
\newcommand{\cross}{\operatorname*{\times}}

\begin{document}

\title[Compact Clifford--Klein forms]{A cohomological obstruction to the existence of compact Clifford--Klein forms} 
\author{Yosuke Morita} 
\address{Graduate School of Mathematical Science, the University of Tokyo, 3-8-1 Komaba, Meguro-ku, Tokyo 153-8914, Japan} 
\email{ymorita@ms.u-tokyo.ac.jp} 
\date{} 

\maketitle

\begin{abstract}
In this paper, we continue the study of the existence problem of 
compact Clifford--Klein forms from a cohomological point of view, 
which was initiated by Kobayashi--Ono 
and extended by Benoist--Labourie and the author. 
We give an obstruction to the existence of compact Clifford--Klein forms 
by relating a natural homomorphism 
from relative Lie algebra cohomology to de Rham cohomology with 
an upper-bound estimate for cohomological dimensions of discontinuous groups. 
From this obstruction, we derive some examples, e.g. 
$\SO_0(p+r, q)/(\SO_0(p,q) \times \SO(r))$ $(p,q,r \geq 1, \ q:\text{odd})$ 
and $\SL(p+q, \C)/\SU(p,q)$ $(p,q \geq 1)$, 
of a homogeneous space that does not admit a compact Clifford--Klein form. 
To construct these examples, we apply H. Cartan's theorem 
on relative Lie algebra cohomology of reductive pairs and 
the theory of $\ep$-families of semisimple symmetric pairs. 
\end{abstract}

\section{Introduction}

\subsection{The existence problem of compact Clifford--Klein forms}

A Clifford--Klein form is a double coset space $\Gamma \bs G/H$, 
where $G$ is a Lie group, $H$ a closed subgroup of $G$, 
and $\Gamma$ a discrete subgroup of $G$ acting properly and freely on $G/H$. 
It admits a natural structure of a manifold locally modelled on $G/H$. 
If $\Gamma \bs G/H$ is a Clifford--Klein form, a discrete subgroup 
$\Gamma$ of $G$ is called a discontinuous group for $G/H$. 

It is one of the central open problems in the study of Clifford--Klein forms 
to determine all homogeneous spaces admitting 
\emph{compact} Clifford--Klein forms. 
In the last three decades, 
this problem attracted considerable attention, and a number of obstructions to 
the existence of compact Clifford--Klein forms were found. 
Some of these obstructions are based on a homomorphism 
\[
\eta : H^\bl(\g, \h; \R) \to H^\bl(\Gamma \bs G/H; \R) \qquad 
(\g = \Lie(G), \ \h = \Lie(H)), 
\]
which imposes a restriction on cohomology of compact Clifford--Klein forms 
(\cite{Kob-Ono90}, \cite{Kob89}, \cite{Ben-Lab92}, \cite{Mor15}, \cite{Mor15+}). 
In this paper, we give a new obstruction arising from this homomorphism. 

\subsection{Main result}

The main result of this paper is as follows: 
\begin{thm}\label{thm:main}
Let $G$ be a connected linear Lie group and $H$ its connected closed subgroup. 
Assume that $H^N(\g, \h; \R) \neq 0 \ (N = \dim G - \dim H)$. 
Let $K_H$ be a maximal compact subgroup of $H$ and 
$T_H$ a maximal torus of $K_H$. 
Let $I^\bl = \bigoplus_{n \in \N} I^n$ 
be the graded ideal of $H^\bl(\g, \t_H; \R)$ generated by 
\[
\bigoplus_{C, \, p}
\im ( i : H^p(\g, \c; \R) \to H^p(\g, \t_H; \R) ), 
\]
where the direct sum runs 
all connected compact subgroups $C$ of $G$ containing $T_H$ and 
all $p > N + \dim K_H - \dim C$. 
If 
\[
\im ( i : H^N(\g, \h; \R) \to H^N(\g, \t_H; \R) ) \subset I^N
\]
holds, $G/H$ does not admit a compact Clifford--Klein form. 
\end{thm}
\begin{rem}
We do not know if Theorem~\ref{thm:main} 
applies to a more general case of manifolds locally modelled on $G/H$. 
\end{rem}
The key to the proof of Theorem~\ref{thm:main} is 
an upper-bound estimate for cohomological dimensions of discontinuous groups 
(Lemma~\ref{lem:dim-estimate}), 
which was established by Kobayashi \cite{Kob89} in the reductive case. 
It imposes another restriction on cohomology of Clifford--Klein forms. 
We prove Theorem~\ref{thm:main} by linking these two restrictions. 
\begin{rem}
In \cite{Kob92Duke}, Kobayashi gave an obstruction to 
the existence of compact Clifford--Klein forms by combining 
the estimate for cohomological dimensions 
with the criterion for proper actions. As far as the author understands, 
his and our obstructions do not include each other. 
\end{rem}

\subsection{New examples of a homogeneous space without compact Clifford--Klein forms}

Theorem~\ref{thm:main} provides some new examples of 
a homogeneous space that does not admit a compact Clifford--Klein form. 
For example, an irreducible symmetric space $G/H$ does not have 
a compact Clifford--Klein form if the corresponding symmetric pair 
$(\g, \h)$ is as in Table~\ref{table:new}
(see Corollaries \ref{cor:C-R}, \ref{cor:hyperbolic} and \ref{cor:so-sym}): 

\begin{table}[!h]
\begin{center}
\begin{tabular}{lll} \toprule
\multicolumn{1}{c}{$\g$} & \multicolumn{1}{c}{$\h$} & \multicolumn{1}{c}{Conditions} \\ \midrule 
$\sl(p+q, \C)$ & $\su(p, q)$ & $p, q \geq 1$ \\ \hline 
$\sl(p+q, \R)$ & $\so(p, q)$ & $p, q \geq 1$ \\ \hline 
$\sl(p+q, \H)$ & $\sp(p, q)$ & $p, q \geq 1$ \\ \hline 
$\so(p+q, \C)$ & $\so(p, q)$ & $p, q \geq 2, \ (p, q) \neq (2, 2)$ \\ \hline 
$\so(2n+1, \C)$ & $\so(2n, 1)$ & $n \geq 1$ \\ \hline
$\so(2n, \C)$ & $\so^\ast(2n)$ & $n \geq 3$ \\ \hline 
$\so(p+r, q)$ & $\so(p,q) \oplus \so(r)$ & $p, q, r \geq 1, \ q : \text{odd}$ \\ \hline 
$\sp(p+q, \C)$ & $\sp(p, q)$ & $p, q \geq 1$ \\ \hline 
$\e_{6, \C}$ & $\e_{6(-14)}$ & --- \\ \hline 
$\e_{6(6)}$ & $\sp(2, 2)$ & --- \\ \hline 
$\e_{7, \C}$ & $\e_{7(-5)}$ & --- \\ \hline 
$\e_{7, \C}$ & $\e_{7(-25)}$ & --- \\ \hline 
$\e_{8, \C}$ & $\e_{8(-24)}$ & --- \\ \hline 
$\f_{4, \C}$ & $\f_{4(-20)}$ & --- \\ 
\bottomrule \\
\end{tabular}
\caption{Irreducible symmetric spaces without compact Clifford--Klein forms} 
\label{table:new}
\end{center}
\end{table}

In particular, the nonexistence of a compact Clifford--Klein form of 
\[
\SO_0(p+1, q)/\SO_0(p,q) \qquad (p,q \geq 1, \ q : \text{odd}) 
\]
is rephrased as: 
\begin{cor}
If $p, q \geq 1$ and $q$ is odd, then there does not exist a compact complete pseudo-Riemannian manifold of signature $(p,q)$ with constant positive sectional curvature. 
\end{cor}
We also obtain some nonsymmetric examples: for example, 
\[
\SL(n, \R) / \SL(m, \R) \qquad (n > m \geq 2, \ m : \text{even})
\]
does not admit a compact Clifford--Klein form (see Corollary~\ref{cor:sl}). 
\begin{rem}
While the author is preparing this manuscript, 
Tholozan \cite{Tho15+} announced the nonexistence of 
compact Clifford--Klein forms of some homogeneous spaces, such as 
\begin{enumerate}
\item[(1)] $\SO_0(p+r, q) / \SO_0(p,q) \quad (p, q, r \geq 1, \ q : \text{odd})$, 
\item[(2)] $\SL(n, \R) / \SL(m, \R) \quad (n > m \geq 2, \ m : \text{even})$, 
\item[(3)] $\SL(p+q, \R) / \SO_0(p,q) \quad (p,q \geq 1, \ p+q : \text{odd})$, 
\item[(4)] $\SO(n, \C) / \SO(m, \C) \quad (n > m \geq 2, \ m : \text{even})$, 
\item[(5)] $\SO(p+q, \C) / \SO_0(p,q) \quad (p,q \geq 1, \ p+q : \text{odd})$, 
\end{enumerate}
Our results are sharper for (3), (5) and the same for (1)--(2), (4). 
Actually, (4) had been proved by earlier methods 
\cite{Kob92Duke}, \cite{Mor15} too, as well as some special cases of 
(1)--(3), (5) as below. 
\end{rem}
\begin{rem}
We mention some related nonexistence results 
which can be obtained by previously known methods: 

\begin{itemize}
\item\textbf{Calabi--Markus \cite{Cal-Mar62}, Wolf \cite{Wol62}, \cite{Wol64}, Kobayashi \cite{Kob89}:} 
The following homogeneous spaces do not admit infinite discontinuous groups. In particular, they do not admit compact Clifford--Klein forms: 
\begin{itemize}
\item $\SO(p+q, \C)/\SO_0(p, q) \quad (p,q \geq 1, \ |p-q| \leq 1)$, 
\item $\SO_0(p+r, q+s)/(\SO_0(p, q) \times \SO_0(r, s))$ \par
$(p \geq q \geq 1, \ r \geq s \geq 0, \ (r,s) \neq (0,0))$. 
\end{itemize}

\item\textbf{Kulkarni \cite{Kul81}:}
A homogeneous space 
\begin{itemize}
\item $\SO_0(p+1, q)/\SO_0(p,q) \quad (p,q \geq 1, \ p,q : \text{odd})$
\end{itemize}
does not admit a compact Clifford--Klein form. 

\item\textbf{Kobayashi \cite{Kob92Lake}, \cite{Kob92Duke}, \cite{Kob97}:} 
The following homogeneous spaces do not admit compact Clifford--Klein forms: 
\begin{itemize}
\item $\SL(2p, \C)/\SU(p, p) \quad (p \geq 1)$, 
\item $\SL(2p, \R)/\SO_0(p, p) \quad (p \geq 1)$, 
\item $\SO_0(p+r, q+s)/(\SO_0(p, q) \times \SO_0(r, s)) \quad (p, q, r, s \geq 1)$, 
\item $\SO_0(p+r, q)/(\SO_0(p, q) \times \SO(r)) \quad (p+r > q \geq 1)$, 
\item $\SL(n, \R)/\SL(m, \R) \quad (n > 3[(m+1)/2], \ m \geq 2)$. 
\end{itemize}

\item\textbf{Zimmer \cite{Zim94}, Labourie--Mozes--Zimmer \cite{LMZ95}, Labourie--Zimmer \cite{Lab-Zim95}:} 
A homogeneous space 
\begin{itemize}
\item $\SL(n, \R) / \SL(m, \R) \quad (n-3 \geq m \geq 2)$ 
\end{itemize}
does not admit a compact Clifford--Klein form. 
If, in addition, $n \geq 2m$, there does not exist 
a compact manifold locally modelled on this homogeneous space. 

\item\textbf{Shalom \cite{Sha00}:} 
A homogeneous space 
\begin{itemize}
\item $\SL(n, \R) / \SL(2, \R) \quad (n \geq 4)$ 
\end{itemize}
does not admit a compact Clifford--Klein form. 

\item\textbf{Margulis \cite{Mar97}, Oh \cite{Oh98}:}
Let $\alpha_n : \SL(2, \R) \to \SL(n, \R)$ 
denote the real $n$-dimensional irreducible representation of $\SL(2, \R)$. 
Then, a homogeneous space 
\begin{itemize}
\item $\SL(n, \R) / \alpha_n(\SL(2, \R)) \quad (n \geq 4)$
\end{itemize}
does not admit a compact Clifford--Klein form.

\item\textbf{Benoist \cite{Ben96}, Okuda \cite{Oku13}:} 
The following homogeneous spaces do not admit non-virtually abelian discontinuous groups. In particular, they do not admit compact Clifford--Klein forms: 
\begin{itemize}
\item $\SL(p+q, \C)/\SU(p,q) \quad (p,q \geq 1, \ |p-q| \leq 1)$, 
\item $\SL(p+q, \R)/\SO_0(p,q) \quad (p,q \geq 1, \ |p-q| \leq 1)$, 
\item $\SL(p+q, \H)/\Sp(p,q) \quad (p,q \geq 1, \ |p-q| \leq 1)$, 
\item $\SO(4p+2, \C)/\SO_0(2p+2, 2p) \quad (p \geq 1)$, 
\item $\SO_0(p+q+1, p+q+1)/(\SO_0(p+1,p) \times \SO_0(q,q+1))$ \par
$(p \geq 1, \ q \geq 0, \ p+q : \text{even})$, 
\item $\SL(m+1, \R) / \SL(m, \R) \quad (m \geq 2, \ m : \text{even})$. 
\end{itemize}

\item\textbf{Kobayashi--Ono \cite{Kob-Ono90}, Kobayashi \cite{Kob89}, Benoist--Labourie \cite{Ben-Lab92}, Morita \cite{Mor15}, \cite{Mor15+}:} 
There do not exist compact manifolds 
locally modelled on the following homogeneous spaces. 
In particular, they do not admit compact Clifford--Klein forms: 
\begin{itemize}
\item $\SL(p+q, \R)/\SO_0(p,q) \quad (p,q \geq 1, \ p,q : \text{odd})$, 
\item $\SO_0(p+r, q+s)/(\SO_0(p,q) \times \SO_0(r,s))$ \par
$(p,q,r \geq 1, \ s \geq 0, \ p,q : \text{odd})$. 
\end{itemize}
\end{itemize}
\end{rem}
\begin{rem}
Now we mention some homogeneous spaces admitting 
compact Clifford--Klein forms: 

\begin{itemize}
\item\textbf{Borel--Harish-Chandra \cite{Bor-HC62}, Mostow--Tamagawa \cite{Mos-Tam62}, Borel \cite{Bor63}:} 
Every Riemannian symmetric space $G/K$ admits a compact Clifford--Klein form. 

\item\textbf{Kulkarni \cite{Kul81}, Kobayashi \cite{Kob92Lake}, \cite{Kob97}, Kobayashi--Yoshino \cite{Kob-Yos05}:} 
The following homogeneous spaces admit compact Clifford--Klein forms. 
\begin{itemize}
\item $\SO(8, \C)/\SO_0(7,1)$, 
\item $\SO_0(p+r, q)/(\SO_0(p,q) \times \SO(r))$ \par
$((p,q,r) = (1,2n,1), (3,4n,1), (1,4,2), (1,4,3), (7,8,1), \, n \geq 1)$.
\end{itemize}
\end{itemize}
\end{rem}

\subsection{Outline of the paper} 

In Section~\ref{sec:prelim}, 
we recall some basic facts on cohomology of Clifford--Klein forms, including 
the upper-bound estimate for cohomological dimensions of discontinuous groups 
and the definition of the homomorphism $\eta$ 
from relative Lie algebra cohomology to de Rham cohomology. 
The proof of Theorem~\ref{thm:main} is given in Section~\ref{sec:proof}. 
The rest of the paper is devoted to construct examples of a homogeneous space 
to which Theorem~\ref{thm:main} is applicable. 
In Section~\ref{sec:Cartan}, we apply H. Cartan's theorem \cite{Car51} on 
relative Lie algebra cohomology of reductive pairs to Theorem~\ref{thm:main}. 
We shall see that, if $G/H$ is a homogeneous space of reductive type 
satisfying some invariant-theoretic condition, 
then Theorem~\ref{thm:main} is applicable to $G/H$. 
In Section~\ref{sec:ep}, we give a way to construct 
semisimple symmetric spaces satisfying the condition obtained in 
Section~\ref{sec:Cartan} by using the theory of $\ep$-families, 
established by Oshima--Sekiguchi \cite{Osh-Sek84}. 
Finally, in Section~\ref{sec:ex}, we give examples of a homogeneous space 
without compact Clifford--Klein forms by applying the results in 
Sections~\ref{sec:Cartan} and \ref{sec:ep}. 

\section{Preliminaries for the proof of Theorem~\ref{thm:main}}
\label{sec:prelim}

We recall some basic facts on cohomology of Clifford--Klein forms. 
They are used in Section~\ref{sec:proof}. 

\subsection{Orientability of Clifford--Klein forms}
\label{subsec:ori}

Let $G$ be a Lie group, $H$ a closed subgroup of $G$, 
and $\Gamma$ a discrete subgroup of $G$ acting properly and freely on $G/H$. 
The local system of orientation of $\Gamma \bs G/H$ is isomorphic to 
$\Gamma \bs G \cross_H \R$, where $H$ acts on $\R$ via 
$H \to \{ \pm 1 \}$, $h \mapsto \sgn \det \Ad_{\g/\h}(h)$. 
Thus, $\Gamma \bs G/H$ is orientable if $H$ is connected. 

\subsection{Maximal compact subgroups of Lie groups}

The following fact is fundamental for 
the computation of cohomology of homogeneous spaces and Clifford--Klein forms: 
\begin{fact}[Cartan--Malcev--Iwasawa--Mostow, {\cite[Ch.~VII, Th.~1.2]{Bor98}}, {\cite[Ch.~XV, Th.~3.1]{Hoc65}}]\label{fact:CMIM}
Let $G$ be a Lie group with finitely many connected components. Then, 
\begin{enumerate}
\item[\textup{(1)}]
Every compact subgroup of $G$ is contained in some maximal compact subgroup. 
\item[\textup{(2)}]
Any two maximal compact subgroups of $G$ are conjugate. 
\item[\textup{(3)}]
Let $K$ be a maximal compact subgroup of $G$. 
Then, there exists linear subspaces $V_1, \dots, V_s$ of $\g$ such that 
\[
V_1 \times \dots \times V_s \times K \to G, 
\qquad (v_1, \dots, v_s, k) \mapsto \exp(v_1) \dots \exp(v_s) k
\]
is a diffeomorphism. 
\end{enumerate}
\end{fact}

\subsection{Cohomological dimensions of discontinuous groups}
\label{subsec:cdR}

Recall that the real cohomological dimension 
$\cdR(\Gamma)$ of a discrete group $\Gamma$ is defined as 
\[
\cdR(\Gamma) = \sup \{ p \in \N : 
\text{$H^p(\Gamma; V) \neq 0$ for some $\R\Gamma$-module $V$} \}. 
\]
Let $G$ be a connected Lie group, $H$ a connected closed subgroup of $G$, 
and $\Gamma$ a torsion-free discrete subgroup of $G$ acting properly 
(and therefore freely) on $G/H$. 
We put 
\[
\cdR(\Gamma; G/H) = \sup \{ p \in \N : 
\text{$H^p(\Gamma \bs G/H; \V) \neq 0$ for some $\R\Gamma$-module $V$} \}, 
\]
where $\V$ denotes the local system $V \times_\Gamma G/H$ on $\Gamma \bs G/H$. 
Remind that $\cdR(\Gamma; G/K) = \cdR(\Gamma)$, 
where $K$ is a maximal compact subgroup of $G$, 
because $G/K$ is a classifying space of $\Gamma$ by Fact~\ref{fact:CMIM}. 

\begin{lem}
\label{lem:dim-estimate}
Let $G$, $H$ and $\Gamma$ be as above. 
Put $N = \dim G - \dim H$. 
Let $K$ and $K_H$ be maximal compact subgroups of $G$ and $H$, respectively. 
Then, 
\begin{enumerate}
\item[\textup{(1)}] 
$\cdR(\Gamma; G/H) \leq N$; 
equality is attained if and only if the Clifford--Klein form 
$\Gamma \bs G/H$ is compact. 
\item[\textup{(2)}] 
$\cdR(\Gamma; G/H) = \cdR(\Gamma) + \dim K - \dim K_H$. 
\end{enumerate}
\end{lem}
\begin{proof}
These are proved in \cite[\S 5]{Kob89} when $G/H$ is of reductive type. 
Our proof is along the same line. 

(1) Since the Clifford--Klein form $\Gamma \bs G/H$ is orientable, 
the Poincar\'e duality for $\Gamma \bs G/H$ is stated as: 
\[
H^p(\Gamma \bs G/H; \V) \simeq H^{\mathrm{BM}}_{N-p}(\Gamma \bs G/H; \V), 
\]
where the right-hand side is the Borel--Moore homology. 
This implies $\cdR(\Gamma; G/H) \leq N$, 
with equality if and only if $\Gamma \bs G/H$ is compact. 

(2) Take any $\R\Gamma$-module $V$. The Cartan--Leray spectral sequence 
\cite[Ch.~XVI, \S 9]{Car-Eil56} for the $\Gamma$-action on $G/H$ is: 
\[
E_2^{p,q} = H^p(\Gamma; H^q(G/H; V)) \Rightarrow H^{p+q}(\Gamma \bs G/H; \V). 
\]
Since $G$ is connected, 
its subgroup $\Gamma$ acts trivially on $H^q(G/H; \R)$. 
The $E_2$-term of the spectral sequence is thus rewritten as: 
\[
E_2^{p,q} = H^p(\Gamma; V) \otimes H^q(G/H; \R). 
\]
Therefore, we have 
\[
\cdR(\Gamma; G/H) = \cdR(\Gamma) + \sup \{ q \in \N : H^q(G/H; \R) \neq 0 \}. 
\]
Note that $\cdR(\Gamma) = \cdR(\Gamma; G/K) < \infty$ by (1). 
On the other hand, 
\begin{lem}\label{lem:homotopy-equiv}
The composition of an inclusion and a projection 
\[
\pi \circ i : K/K_H \to G/K_H \to G/H
\]
is a homotopy equivalence. 
\end{lem}
\begin{proof}
It directly follows from Fact~\ref{fact:CMIM} that 
the inclusion $i$ is a homotopy equivalence. 
The projection $\pi$ is a fibre bundle 
whose typical fibre $H/K_H$ is contractible again by Fact~\ref{fact:CMIM}, 
hence a homotopy equivalence (\cite[Cor.~3.2]{Dol63}). 
\end{proof}
Since $K/K_H$ is an orientable compact manifold, 
we obtain from Lemma~\ref{lem:homotopy-equiv} that 
\begin{align*}
\sup \{ q \in \N : H^q(G/H; \R) \neq 0 \} 
&= \sup \{ q \in \N : H^q(K/K_H; \R) \neq 0 \} \\
&= \dim K - \dim K_H. \qedhere
\end{align*}
\end{proof}

\subsection{A homomorphism $\eta$ from relative Lie algebra cohomology to de Rham cohomology}
\label{subsec:eta}

Let $G$ be a Lie group, $H$ a connected closed subgroup of $G$, 
and $\Gamma$ a discrete subgroup of $G$ acting properly and freely on $G/H$. 
We define $\eta : H^\bl(\g, \h; \R) \to H^\bl(\Gamma \bs G/H; \R)$ 
to be the homomorphism induced from the inclusion map
\[
(\Lambda (\g/\h)^\ast)^\h \simeq \Omega(G/H)^G \hookrightarrow 
\Omega(G/H)^\Gamma \simeq \Omega(\Gamma \bs G/H). 
\]
If a Clifford--Klein form $\Gamma \bs G/H$ is compact, 
\[
\eta : H^N(\g, \h; \R) \to H^N(\Gamma \bs G/H; \R) \qquad (N = \dim G - \dim H)
\]
is injective (\cite[\S 3]{Mor15+}). 
Indeed, if $\Phi \in (\Lambda^N (\g/\h)^\ast)^\h$ is nonzero, 
then $\eta([\Phi]) \in H^N(\Gamma \bs G/H; \R)$ is 
a cohomology class of a volume form, hence nonzero. 

\section{Proof of Theorem~\ref{thm:main}}
\label{sec:proof}

Suppose there were a discrete subgroup $\Gamma$ of $G$ such that 
$\Gamma \bs G/H$ is a compact Clifford--Klein form. 
Such $\Gamma$ is always finitely generated (\cite[Lem.~2.1]{Kob89}). 
By Selberg's lemma \cite[Lem.~8]{Sel60}, we can assume $\Gamma$ 
is torsion-free without loss of generality. We shall see that 
\[
\eta \circ i : 
H^N(\g, \h; \R) \to H^N(\g, \t_H; \R) \to H^N(\Gamma \bs G/T_H; \R)
\]
is a zero map and injective, which is impossible. 

Let $C$ be any compact connected subgroup of $G$ containing $T_H$. 
Since $\Gamma$ is torsion-free, $\Gamma \bs G/C$ is a Clifford--Klein form. 
By Lemma~\ref{lem:dim-estimate}, we have 
\begin{align*}
\cdR(\Gamma; G/C) 
&= \cdR(\Gamma) + \dim K - \dim C \\ 
&= \cdR(\Gamma; G/H) + \dim K_H - \dim C \\ 
&= N + \dim K_H - \dim C. 
\end{align*}
From the commutativity of the diagram 
\[
\xymatrix{
H^p(\g, \c; \R) \ar[r]^{i} \ar[d]_{\eta} & 
H^p(\g, \t_H; \R) \ar[d]^{\eta} \\
H^p(\Gamma \bs G/C; \R) \ar[r]^{\pi^\ast} & 
H^p(\Gamma \bs G/T_H; \R), 
}
\]
it follows that 
\[
\eta \circ i : 
H^p(\g, \c; \R) \to H^p(\g, \t_H; \R) \to H^p(\Gamma \bs G/T_H; \R)
\]
is a zero map for $p > N + \dim K_H - \dim C$. Therefore 
\[
I^\bl \subset 
\ker (\eta : H^\bl(\g, \t_H; \R) \to H^\bl(\Gamma \bs G/T_H; \R)). 
\]
In particular, 
\[
\eta \circ i : 
H^N(\g, \h; \R) \to H^N(\g, \t_H; \R) \to H^N(\Gamma \bs G/T_H; \R)
\]
is a zero map because 
$\im (i : H^N(\g, \h; \R) \to H^N(\g, \t_H; \R)) \subset I^N$. 

Consider another commutative diagram 
\[
\xymatrix{
H^N(\g, \h; \R) \ar[rr]^{i} \ar[d]_{\eta} & 
& 
H^N(\g, \t_H; \R) \ar[d]^{\eta} \\
H^N(\Gamma \bs G/H; \R) \ar[r]^{\pi^\ast} & 
H^N(\Gamma \bs G/K_H; \R) \ar[r]^{\pi^\ast} & 
H^N(\Gamma \bs G/T_H; \R). 
}
\]
As we recalled in Subsection~\ref{subsec:eta}, 
\[
\eta : H^N(\g, \h; \R) \to H^N(\Gamma \bs G/H; \R)
\]
is injective. On the other hand, the projections 
$\pi : \Gamma \bs G/K_H \to \Gamma \bs G/H$ and 
$\pi : \Gamma \bs G/T_H \to \Gamma \bs G/K_H$ 
are fibre bundles with typical fibres $H/K_H$ and $K_H/T_H$, 
respectively. The induced homomorphism on cohomology
\[
\pi^\ast : H^N(\Gamma \bs G/H; \R) \to H^N(\Gamma \bs G/K_H; \R) 
\]
is isomorphic since $H/K_H$ is contractible (Fact~\ref{fact:CMIM}), and 
\[
\pi^\ast : H^N(\Gamma \bs G/K_H; \R) \to H^N(\Gamma \bs G/T_H; \R) 
\]
is injective by the splitting principle (\cite[Th.~6.8.3]{Gui-Ste99}). 
Thus, the composition map 
\[
\eta \circ i : 
H^N(\g, \h; \R) \to H^N(\g, \t_H; \R) \to H^N(\Gamma \bs G/T_H; \R)
\]
is injective. 
This completes the proof of Theorem~\ref{thm:main}. \qed

\section{A sufficient condition for Theorem~\ref{thm:main} in the reductive case}\label{sec:Cartan}

\subsection{Cartan's theorem}

We say that $(\g, \h)$ is a real (resp. complex) reductive pair 
if $\g$ is a real (resp. complex) reductive Lie algebra and 
$\h$ is a real (resp. complex) subalgebra of $\g$ that is reductive in $\g$. 
In this paper, we say that a homogeneous space $G/H$ is of reductive type 
if $G$ is a connected linear Lie group and 
$H$ is a connected closed subgroup of $G$ such that 
$(\g, \h)$ is a real reductive pair. 

Relative Lie algebra cohomology of real or complex reductive pairs 
can be easily computed by H. Cartan's theorem \cite{Car51}. 
Let us briefly recall the statement of the theorem 
(see \cite{GHV76}, \cite{Oni94} for details). 
Let $(\g, \h)$ be a real or complex reductive pair. 
Let $P \g^\ast = \bigoplus_{n \geq 1} P^{2n-1} \g^\ast$ 
be the primitive subspace of 
$(\Lambda \g^\ast)^\g$ (\cite[Ch.~V, \S 5]{GHV76}). 
The inclusion $P \g^\ast \subset (\Lambda \g^\ast)^\g$
induces an isomorphism $\Lambda (P \g^\ast) \simeq (\Lambda \g^\ast)^\g$. 
Fix a transgression 
$\tau : P^{2n-1} \g^\ast \to (S^n \g^\ast)^\g$ in the Weil algebra of $\g$ 
(\cite[Ch.~VI, \S 4]{GHV76}). 
We introduce a grading on an algebra 
$(S \h^\ast)^\h \otimes (\Lambda \g^\ast)^\g$ by 
\[
\deg (Q \otimes \alpha) = 2 \deg Q + \deg \alpha 
\qquad (Q \in (S \h^\ast)^\h, \ \alpha \in (\Lambda \g^\ast)^\g)
\]
and define a differential $\delta$ 
on $(S \h^\ast)^\h \otimes (\Lambda \g^\ast)^\g$ by 
\[
\delta (Q \otimes 1) = 0, \quad 
\delta (1 \otimes \alpha) = -\tau(\alpha)|_\h \otimes 1 \qquad 
(Q \in (S \h^\ast)^\h, \ \alpha \in P \g^\ast). 
\]
Cartan constructed a quasi-isomorphism of differential graded algebras 
(i.e. a homomorphism that induces isomorphism on cohomology)
\[
\phi : ( (S \h^\ast)^\h \otimes (\Lambda \g^\ast)^\g, \delta ) 
\to ( (\Lambda (\g/\h)^\ast)^\h, d ) 
\]
(\cite[Ch.~X, \S 2]{GHV76}). 
This $\phi$ is functorial in $\h$, namely, a diagram
\[
\xymatrix{
((S \h^\ast)^\h \otimes (\Lambda \g^\ast)^\g, \delta) \ar[r]^{\ \ \phi} \ar[d]_{\rest \otimes 1} &
((\Lambda (\g/\h)^\ast)^\h, d) \ar[d]^{i} \\
((S \mfl^\ast)^\mfl \otimes (\Lambda \g^\ast)^\g, \delta) \ar[r]^{\ \ \phi} &
((\Lambda (\g/\mfl)^\ast)^\mfl, d) 
}
\]
commutes for any subalgebra $\mfl$ of $\h$ that is reductive in $\g$, where 
$\rest : (S \h^\ast)^\h \to (S \mfl^\ast)^\mfl$ denotes the restriction map. 

\subsection{A sufficient condition for Theorem~\ref{thm:main} in terms of invariants}

\begin{prop}\label{prop:simplification-inv-poly}
A homogeneous space $G/H$ of reductive type 
satisfies the assumptions of Theorem~\ref{thm:main} 
(and therefore does not admit a compact Clifford--Klein form) 
if there exist a compact subgroup $C$ of $G$ and 
a homomorphism of graded algebras 
$\phi : (S \h_\C^\ast)^{\h_\C} \to (S \c_\C^\ast)^{\c_\C}$ such that 
\begin{enumerate}
\item[\textup{(i)}] $\dim C > \dim K_H$, 
\item[\textup{(ii)}] $C$ contains a maximal torus $T_H$ of $K_H$, and
\item[\textup{(iii)}] the diagram 
\[
\xymatrix{
(S \g_\C^\ast)^{\g_\C} \ar[r]^{\rest} \ar[rd]_{\rest} & 
(S \h_\C^\ast)^{\h_\C} \ar[r]^{\rest} \ar[d]^{\phi} & 
S (\t_H)_\C^\ast \\
& 
(S \c_\C^\ast)^{\c_\C} \ar[ru]_{\rest} & 
}
\]
commutes. 
\end{enumerate}
\end{prop}
\begin{proof}
It suffices to see that 
$\im ( i : H^N(\g, \h; \R) \to H^N(\g, \t_H; \R) ) \subset I^N$. 
By (iii), 
\[
\xymatrix{
( (S \h_\C^\ast)^{\h_\C} \otimes (\Lambda \g_\C^\ast)^{\g_\C}, \delta ) \ar[rr]^{\rest \otimes 1} \ar[d]_{\phi \otimes 1} & & 
( S (\t_H)_\C^\ast \otimes (\Lambda \g_\C^\ast)^{\g_\C}, \delta ) \\
( (S \c_\C^\ast)^{\c_\C} \otimes (\Lambda \g_\C^\ast)^{\g_\C}, \delta ) \ar[rru]_{\rest \otimes 1} & & 
}
\]
is a commutative diagram of differential graded algebras. 
The induced commutative diagram on cohomology
\[
\xymatrix{
H^\bl(\g_\C, \h_\C; \C) \ar[rr]^{i} \ar[d]_{\phi \otimes 1} & & 
H^\bl(\g_\C, (\t_H)_\C; \C) \\
H^\bl(\g_\C, \c_\C; \C) \ar[rru]_{i} & &
}
\]
implies 
\begin{align*}
&\im ( i : H^N(\g_\C, \h_\C; \C) \to H^N(\g_\C, (\t_H)_\C; \C) ) \\
& \quad \subset 
\im ( i : H^N(\g_\C, \c_\C; \C) \to H^N(\g_\C, (\t_H)_\C; \C) ), 
\end{align*}
or equivalently, 
\[
\im ( i : H^N(\g, \h; \R) \to H^N(\g, \t_H; \R) ) \subset 
\im ( i : H^N(\g, \c; \R) \to H^N(\g, \t_H; \R) ), 
\]
while 
\[
\im ( i : H^N(\g, \c; \R) \to H^N(\g, \t_H; \R) ) \subset I^N
\]
by (i). This completes the proof. 
\end{proof}

\section{The case of semisimple symmetric spaces}
\label{sec:ep}

\subsection{Semisimple symmetric pairs}

Let $\g$ be a real semisimple Lie algebra and $\sigma$ an involution of $\g$. 
Let $\h = \g^\sigma$ and $\q = \g^{-\sigma}$ be 
the fixed point sets of $\sigma$ and $-\sigma$, respectively. 
We call $(\g, \h)$ a semisimple symmetric pair. 
We say that $(\g, \h)$ is irreducible if $\g$ is simple or 
$(\g, \h) = (\mfl \oplus \mfl, \Delta \mfl)$ for some real simple Lie algebra $\mfl$. 
Every semisimple symmetric pair can be uniquely written 
as a direct sum of irreducible ones. 

Take a Cartan involution $\theta$ of $\g$ 
such that $\theta \sigma = \sigma \theta$. 
Put $\k = \g^\theta$ and $\p = \g^{-\theta}$. 
We have a direct sum decomposition 
$\g = \k \cap \h \oplus \k \cap \q \oplus \p \cap \h \oplus \p \cap \q$. 
Let $\a$ be a maximal abelian subspace of $\p \cap \q$. 
For $\alpha \in \a^\ast$, we put 
\[
\g_\alpha = \{ X \in \g : \text{$[Y,X] = \alpha(Y)X$ for any $Y \in \a$} \}. 
\]
Then 
$\Sigma = \{ \alpha \in \a^\ast : \g_\alpha \neq 0 \} \smallsetminus \{ 0 \}$ 
satisfies the axioms of root system (\cite[Th.~5]{Ros79}). 
We call $\Sigma$ the restricted root system of $(\g, \h)$. 
If $\h = \k$, then $\Sigma$ is nothing but 
the restricted root system of the real semisimple Lie algebra $\g$. 
We fix a simple system $\Psi$ of $\Sigma$ 
and write $\Sigma^+$ for the set of positive roots with respect to $\Psi$. 

\subsection{$\ep$-families of semisimple symmetric pairs}

Let us review the notion of an $\ep$-family of semisimple symmetric pairs, 
which was introduced by Oshima--Sekiguchi \cite{Osh-Sek84}. 
A map $\ep : \Sigma \to \{ \pm 1 \}$ is 
called a signature of $\Sigma$ if it satisfies 
\begin{itemize}
\item $\ep(-\alpha) = \ep(\alpha)$ 
for any $\alpha \in \Sigma$, and 
\item $\ep(\alpha)\ep(\beta) = \ep(\alpha + \beta)$ 
for any $\alpha, \beta \in \Sigma$ with $\alpha + \beta \in \Sigma$. 
\end{itemize}
Given a signature $\ep$ of $\Sigma$, 
we define an involution $\sigma_\ep$ of $\g$ by 
\[
\sigma_\ep(X) = 
\begin{cases}
\sigma(X) \quad & (X \in \z_\g(\a)), \\
\ep(\alpha)\sigma(X) \quad & (X \in \g_\alpha, \ \alpha \in \Sigma). 
\end{cases}
\]
We write $\h_\ep = \g^{\sigma_\ep}$ and $\q_\ep = \g^{-\sigma_\ep}$. 
It is easily checked that $\sigma_\ep$ commutes with $\sigma$ and $\theta$, 
and $\a$ is a maximal abelian subspace of $\p \cap \q_\ep$. 
Thus, $\Sigma$ is also a restricted root system of 
the semisimple symmetric pair $(\g, \h_\ep)$. 
The set 
$F((\g, \h)) = \{ (\g, \h_\ep) : \text{$\ep$ is a signature of $\Sigma$} \}$ 
is called an $\ep$-family of semisimple symmetric pairs 
(\cite[\S 6]{Osh-Sek84}). 

Let $\alpha \in \Sigma$. 
Since the involution $\theta \sigma$ leaves $\g_\alpha$ invariant, 
we have a direct sum decomposition 
$\g_\alpha = \g_\alpha^+ \oplus \g_\alpha^-$, 
where $\g_\alpha^\pm$ are the eigenspaces of $\theta \sigma$ 
with eigenvalues $\pm 1$, respectively. 
Put $m^\pm(\alpha ; \h) = \dim \g_\alpha^\pm$. 
Note that $m^\pm(\alpha ; \h) = m^\pm(-\alpha ; \h)$. 
If $\ep$ is a signature of $\Sigma$, we have 
\[
m^\pm(\alpha ; \h_\ep) = 
\begin{cases}
m^\pm(\alpha ; \h) \quad &\text{if $\ep(\alpha) = 1$}, \\
m^\mp(\alpha ; \h) \quad &\text{if $\ep(\alpha) = -1$}. 
\end{cases}
\]
A semisimple symmetric pair $(\g, \h)$ is said to be basic if 
$m^+(\alpha ; \h) \geq m^-(\alpha ; \h)$ 
for any $\alpha \in \Sigma$ with $\alpha/2 \notin \Sigma$ 
(\cite[Def.~6.4]{Osh-Sek84}). 
A typical example of a basic pair is a Riemannian symmetric pair $(\g, \k)$. 
For any $\ep$-family $F$ of semisimple symmetric pairs, 
there exists a basic pair in $F$ unique up to isomorphism 
(\cite[Prop.~6.5]{Osh-Sek84}). 

\subsection{A characterization of the basic pairs}

The following result should be known to experts, 
but we give a proof for the sake of completeness. 
\begin{lem}\label{lem:basic-cpt}
If a semisimple symmetric pair $(\g, \h)$ is basic, an inequality 
$\dim(\k \cap \h) \geq \dim(\k \cap \h_\ep)$ 
holds for any signature $\ep$ of $\Sigma$. 
Equality is attained if and only if $(\g, \h_\ep)$ is also basic. 
\end{lem}
\begin{proof}
There is a direct sum decomposition 
\[
\k \cap \h = \z_\g(\a) \cap \k \cap \h \oplus 
\bigoplus_{\alpha \in \Sigma^+} \{ X + \sigma(X) : X \in \g_\alpha^+ \}. 
\]
Hence, 
\[
\dim(\k \cap \h) = \dim(\z_\g(\a) \cap \k \cap \h) 
+ \sum_{\alpha \in \Sigma^+} m^+(\alpha ; \h). 
\]
Similarly, 
\begin{align*}
\dim(\k \cap \h_\ep) &= \dim(\z_\g(\a) \cap \k \cap \h_\ep) + 
\sum_{\alpha \in \Sigma^+} m^+(\alpha ; \h_\ep) \\ 
&= \dim(\z_\g(\a) \cap \k \cap \h) + 
\sum_{\substack{\alpha \in \Sigma^+, \\ \ep(\alpha) = 1}} m^+(\alpha ; \h) + 
\sum_{\substack{\alpha \in \Sigma^+, \\ \ep(\alpha) = -1}} m^-(\alpha ; \h). 
\end{align*}
Notice that $\alpha/2 \notin \Sigma$ if $\ep(\alpha) = -1$. 
Since $(\g, \h)$ is basic, we have 
\[
\dim(\k \cap \h) - \dim (\k \cap \h_\ep) 
= \sum_{\substack{\alpha \in \Sigma^+, \\ \ep(\alpha) = -1}} 
(m^+(\alpha ; \h) - m^-(\alpha ; \h)) \geq 0. 
\]
If equality is attained, then $m^+(\alpha ; \h) = m^-(\alpha ; \h)$ 
for any $\alpha \in \Sigma^+$ with $\ep(\alpha) = -1$. 
This implies that $m^\pm(\alpha ; \h_\ep) = m^\pm(\alpha ; \h)$ 
for any $\alpha \in \Sigma$, thus $(\g, \h_\ep)$ is basic. 
Conversely, if $(\g, \h_\ep)$ is basic, equality is clearly attained. 
\end{proof}

\subsection{Half-signatures}

We say that $\halfep : \Sigma \to \{ \pm 1, \pm \sqrt{-1} \}$ is 
a half-signature of $\Sigma$ if it satisfies 
\begin{itemize}
\item $\halfep(-\alpha) = \halfep(\alpha)^{-1}$ 
for any $\alpha \in \Sigma$, and 
\item $\halfep(\alpha)\halfep(\beta) = \halfep(\alpha + \beta)$ 
for any $\alpha, \beta \in \Sigma$ with $\alpha + \beta \in \Sigma$. 
\end{itemize}
We remark that any map 
$\Psi \to \{ \pm 1, \pm \sqrt{-1} \}$ (resp. $\Psi \to \{ \pm 1 \})$ 
is uniquely extended to a half-signature (resp.~signature) of $\Sigma$. 
Hence, for each signature $\ep$ of $\Sigma$, there exist 
$2^r$ half-signatures $\halfep$ such that $\halfep^2 = \ep$ ($r = \dim \a$). 
Given a half-signature $\halfep$ of $\Sigma$, 
we define an automorphism $f_{\halfep}$ of $\g_\C$ by 
\[
f_{\halfep}(X) = 
\begin{cases}
X \quad & (X \in (\z_\g(\a))_\C), \\
\halfep(\alpha) X \quad & (X \in (\g_\alpha)_\C, \ \alpha \in \Sigma), 
\end{cases}
\]
and put $\g_{\halfep} = \{ X \in \g : f_{\halfep}(X) = X \}$. 
\begin{lem}[cf.~{\cite[Lem.~1.3]{Osh-Sek80}}]\label{lem:f-hat-ep}
Let $\ep$ be a signature of $\Sigma$ and 
$\halfep$ a half-signature of $\Sigma$ such that $\halfep^2 = \ep$. Then, 
\begin{enumerate}
\item[\textup{(1)}] 
$f_{\halfep}(\h_\C) = (\h_\ep)_\C$. 
\item[\textup{(2)}] 
$\h \cap \g_{\halfep} = \h_\ep \cap \g_{\halfep}$. 
\end{enumerate}
\end{lem}
\begin{proof}
These immediately follows from 
\begin{align*}
\h &= \z_\g(\a) \cap \h \oplus 
\bigoplus_{\alpha \in \Sigma^+} \{ X + \sigma(X) : X \in \g_\alpha \}, \\
\h_\ep &= \z_\g(\a) \cap \h \oplus 
\bigoplus_{\substack{\alpha \in \Sigma^+, \\ \ep(\alpha) = 1}} 
\{ X + \sigma(X) : X \in \g_\alpha \} \oplus 
\bigoplus_{\substack{\alpha \in \Sigma^+, \\ \ep(\alpha) = -1}} 
\{ X - \sigma(X) : X \in \g_\alpha \}, \\
\g_{\halfep} &= \z_\g(\a) \oplus 
\bigoplus_{\substack{\alpha \in \Sigma^+, \\ \halfep(\alpha) = 1}} 
\{ X + \sigma(Y) : X, Y \in \g_\alpha \}. \qedhere
\end{align*}
\end{proof}

\subsection{Semisimple symmetric spaces}

Let $G$ be a connected real linear semisimple Lie group 
whose Lie algebra is $\g$. 
Suppose that the involution $\theta$ of $\g$ lifts to $G$, 
and let $H$ be an open subgroup of $G^\theta$. 
The homogeneous space $G/H$ is called a semisimple symmetric space 
associated with the semisimple symmetric pair $(\g, \h)$. 
Let $\theta$ be a Cartan involution of $G$ such that 
$\theta \sigma = \sigma \theta$. Then $K=G^\theta$ and $K_H = K \cap H$ are 
maximal compact subgroups of $G$ and $H$, respectively. 
Let $\ep$ be a signature of the restricted root system $\Sigma$ of $(\g, \h)$. 
The involution $\sigma_\ep$ of $\g$ lifts to $G$ (\cite[Lem.~1.6]{Shi96}). 
Let $H_\ep$ be an open subgroup of $G^{\sigma_\ep}$. 
Then $K_{H_\ep} = K \cap H_\ep$ is a maximal compact subgroup of $H_\ep$. 

\subsection{A sufficient condition for Proposition~\ref{prop:simplification-inv-poly} in terms of $\ep$-families}

Now, we can prove: 
\begin{prop}\label{prop:simplification-ep}
Let $G/H$ a semisimple symmetric space such that $(\g, \h)$ is basic. 
Let $\ep$ be a signature of $\Sigma$ such that $(\g, \h_\ep)$ is not basic. 
Let $\halfep$ be a half-signature of $\Sigma$ such that $\halfep^2 = \ep$. 
If $\rank(\k \cap \h_\ep) = \rank(\k \cap \h_\ep \cap \g_{\halfep})$, 
then $G/H_\ep$ satisfies the assumptions of 
Proposition~\ref{prop:simplification-inv-poly} 
(and therefore does not admit a compact Clifford--Klein form). 
\end{prop}
\begin{proof}
Let 
\[
(f_{\halfep}|_{\h_\C})^\ast : 
(S (\h_\ep)_\C^\ast)^{(\h_\ep)_\C} \simto (S \h_\C^\ast)^{\h_\C} 
\]
be the isomorphism induced by 
$f_{\halfep}|_{\h_\C} : \h_\C \simto (\h_\ep)_\C$ 
(Lemma~\ref{lem:f-hat-ep}~(1)). 
Take a maximal torus $T_{H_\ep}$ of $K \cap H_\ep$ such that 
$\t_{H_\ep} \subset \k \cap \h_\ep \cap \g_{\halfep}$, 
which exists by the rank assumption. 
By Lemma~\ref{lem:f-hat-ep}~(2), we have $T_{H_\ep} \subset K \cap H$. 
Consider the following diagram: 
\[
\xymatrix{
& (S (\h_\ep)_\C^\ast)^{(\h_\ep)_\C} \ar[rd]^{\rest} \ar[d]^{(f_{\halfep}|_{\h_\C})^\ast} & \\ 
(S \g_\C^\ast)^{\g_\C} \ar[ru]^{\rest} \ar[rd]_{\rest} & 
(S \h_\C^\ast)^{\h_\C} \ar[d]^{\rest} & 
S (\t_{H_\ep})_\C^\ast \\ 
& (S (\k \cap \h)_\C^\ast)^{(\k \cap \h)_\C} \ar[ru]_{\rest} & 
}
\]
The right-hand side triangle clearly commutes since 
$f_{\halfep}|_{(\t_{H_\ep})_\C}$ is the identity map of $(\t_{H_\ep})_\C$. 
Take a Cartan subalgebra $\j$ of $\g$ containing $\a$. 
The restriction map $\rest : (S \g_\C^\ast)^{\g_\C} \to S \j_\C^\ast$ 
is injective by Chevalley's restriction theorem 
(\cite[Ch.~VIII, \S 8, no.~3, Th.~1]{BouLie7-9}), 
while $f_{\halfep}|_{\j_\C}$ is the identity map of $\j_\C$. 
This shows that 
$f_{\halfep}^\ast : (S \g_\C^\ast)^{\g_\C} \simto (S \g_\C^\ast)^{\g_\C}$ 
is the identity map, and therefore the left-hand side triangle commutes. 
We conclude by Lemma~\ref{lem:basic-cpt} that $G/H_\ep$ 
satisfies the assumptions of Proposition~\ref{prop:simplification-inv-poly}. 
\end{proof}

\section{Examples}
\label{sec:ex}

\subsection{Nonbasic semisimple symmetric spaces of type $(C, R)$}
\label{subsec:ex-C-R}

Let us consider a semisimple symmetric pair $(\g, \h)$ 
such that $\g$ is a complex semisimple Lie algebra and $\h$ its real form 
(this case is called ``type $(C, R)$'' in \cite{Osh-Sek84}). 
Let $\ep$ be a signature of a restricted root system $\Sigma$ of $(\g, \h)$ 
and $\halfep$ any half-signature such that $\halfep^2 = \ep$. 
It is easy to check that $\h_\ep$ is a real form of $\g$. 
In this case, $\sqrt{-1} \a$ is a maximal abelian subspace of 
$\k \cap \h_\ep = \sqrt{-1} \p \cap \q_\ep$ 
and is contained in $\k \cap \h_\ep \cap \g_{\halfep}$. This implies 
$\rank(\k \cap \h_\ep) = \rank(\k \cap \h_\ep \cap \g_{\halfep})$. 
From Proposition~\ref{prop:simplification-ep}, we conclude: 
\begin{cor}\label{cor:C-R}
Let $G$ be a connected complex semisimple Lie group and $H$ its real form. 
If the semisimple symmetric pair $(\g, \h)$ is not basic, 
$G/H$ does not admit a compact Clifford--Klein form. 
\end{cor}
For the reader's convenience, 
we list in Table~\ref{table:C-R} the nonbasic pair $(\g, \h)$ 
such that $\g$ is a simple complex Lie algebra and $\h$ its real form 
(cf. \cite[\S\S 1, 6]{Osh-Sek84}). 
The sign $\star$ in Table~\ref{table:C-R} signifies that 
the nonexistence of a compact Clifford--Klein form seems to be a new result. 

\begin{table}[!h]
\begin{center}
\begin{tabular}{llll} \toprule
& \multicolumn{1}{c}{$\g$} & \multicolumn{1}{c}{$\h$} & \multicolumn{1}{c}{Conditions} \\ \midrule 
& $\sl(2n, \C)$ & $\sl(2n, \R)$ & $n \geq 1$ \\ \hline 
$\star$ & $\sl(p+q, \C)$ & $\su(p, q)$ & $p, q \geq 1$ \\ \hline 
$\star$ & $\so(p+q, \C)$ & $\so(p, q)$ & $p, q \geq 2, \ (p, q) \neq (2, 2)$ \\ \hline 
$\star$ & $\so(2n+1, \C)$ & $\so(2n, 1)$ & $n \geq 1$ \\ \hline
$\star$ & $\so(2n, \C)$ & $\so^\ast(2n)$ & $n \geq 3$ \\ \hline 
& $\sp(n, \C)$ & $\sp(n, \R)$ & $n \geq 1$ \\ \hline 
$\star$ & $\sp(p+q, \C)$ & $\sp(p, q)$ & $p, q \geq 1$ \\ \hline 
& $\e_{6, \C}$ & $\e_{6(6)}$ & --- \\ \hline 
& $\e_{6, \C}$ & $\e_{6(2)}$ & --- \\ \hline 
$\star$ & $\e_{6, \C}$ & $\e_{6(-14)}$ & --- \\ \hline 
& $\e_{7, \C}$ & $\e_{7(7)}$ & --- \\ \hline 
$\star$ & $\e_{7, \C}$ & $\e_{7(-5)}$ & --- \\ \hline 
$\star$ & $\e_{7, \C}$ & $\e_{7(-25)}$ & --- \\ \hline 
& $\e_{8, \C}$ & $\e_{8(8)}$ & --- \\ \hline 
$\star$ & $\e_{8, \C}$ & $\e_{8(-24)}$ & --- \\ \hline 
& $\f_{4, \C}$ & $\f_{4(4)}$ & --- \\ \hline 
$\star$ & $\f_{4, \C}$ & $\f_{4(-20)}$ & --- \\ \hline 
& $\g_{2, \C}$ & $\g_{2(2)}$ & --- \\ 
\bottomrule \\
\end{tabular}
\caption{Irreducible symmetric pairs $(\g, \h)$ to which Corollary~\ref{cor:C-R} is applicable}
\label{table:C-R}
\end{center}
\end{table}

\subsection{Half-signatures arising from hyperbolic elements}
\label{subsec:ex-hyperbolic}

Let $(\g, \h)$ be a semisimple symmetric pair. 
As before, let $\sigma$ be an involution of $\g$ corresponding to $\h$ and 
$\theta$ a Cartan involution of $\g$ commuting with $\sigma$. 
We define 
$\h^a = \g^{\sigma \theta} \ (= \k \cap \h \oplus \p \cap \q)$ and 
$\q^a = \g^{-\sigma \theta} \ (= \k \cap \q \oplus \p \cap \h)$. 
The semisimple symmetric pair 
$(\g, \h^a)$ is called the associated pair of $(\g, \h)$. 
\begin{cor}\label{cor:hyperbolic}
A semisimple symmetric space $G/H$ 
does not admit a compact Clifford--Klein form if 
$\h^a = \z_\g(X_0)$ for some $X_0 \in \p \smallsetminus \{ 0 \}$. 
\end{cor}

\begin{proof}
Let $\a$ be a maximal abelian subspace of $\p$ containing $X_0$ and 
$\Sigma$ the restricted root system of $\g$ with respect to $\a$. 
We have a direct sum decomposition $\q^a = \q^a_+ \oplus \q^a_-$, where 
\[
\q^a_+ = 
\bigoplus_{\substack{\alpha \in \Sigma, \\ \alpha(X_0) > 0}} \g_\alpha, \qquad 
\q^a_- = 
\bigoplus_{\substack{\alpha \in \Sigma, \\ \alpha(X_0) < 0}} \g_\alpha. 
\]
It is easily checked that 
$[\q^a_+, \q^a_+] = [\q^a_-, \q^a_-] = 0, \ [\q^a_+, \q^a_-] \subset \h^a$. 
Hence, a map 
\[
\halfep : \Sigma \to \{ \pm 1, \pm \sqrt{-1} \}, \qquad 
\alpha \mapsto 
\begin{cases}
0 \quad &\text{if} \ \alpha(X_0) = 0, \\ 
\sqrt{-1} \quad &\text{if} \ \alpha(X_0) > 0, \\ 
-\sqrt{-1} \quad &\text{if} \ \alpha(X_0) < 0. 
\end{cases}
\]
is a half-signature of $\Sigma$. Put $\ep = \halfep^2$. 
By construction, we have $\h = \k_\ep$ and $\h^a = \g_{\halfep}$. 
Thus $\k \cap \k_\ep = \k \cap \k_\ep \cap \g_{\halfep} = \k \cap \h$. 
Now, Corollary~\ref{cor:hyperbolic} follows from 
Proposition~\ref{prop:simplification-ep} 
if we could prove that $(\g, \h)$ is not basic. 
If $(\g, \h)$ is basic, i.e. $\h$ is isomorphic to $\k$, 
then $\sigma$ and $\theta$ are two commuting Cartan involutions of $\g$, 
hence $\sigma = \theta$ (\cite[Proof of Cor.~6.19]{Kna02}). 
Then $\g = \h^a = \z_\g(X_0)$, which is absurd. 
\end{proof}
For the reader's convenience, we list in Table~\ref{table:hyperbolic} 
the irreducible symmetric pairs $(\g, \h)$ such that 
$\h^a = \z_\g(X_0)$ for some $X_0 \in \p \smallsetminus \{ 0 \}$ 
(cf. \cite[\S 1]{Osh-Sek84}). 
The sign $\star$ in Table~\ref{table:hyperbolic} signifies that 
the nonexistence of a compact Clifford--Klein form seems to be a new result. 
We remark that some examples such as $(\sl(p+q, \C), \su(p,q)) \ (p,q \geq 1)$ 
appear in both Table~\ref{table:C-R} and Table~\ref{table:hyperbolic}.

\begin{table}[!h]
\begin{center}
\begin{tabular}{llll} \toprule
& \multicolumn{1}{c}{$\g$} & \multicolumn{1}{c}{$\h$} & \multicolumn{1}{c}{Conditions} \\ \midrule 
$\star$ & $\sl(p+q, \C)$ & $\su(p, q)$ & $p, q \geq 1$ \\ \hline 
$\star$ & $\sl(p+q, \R)$ & $\so(p, q)$ & $p, q \geq 1$ \\ \hline 
$\star$ & $\sl(p+q, \H)$ & $\sp(p, q)$ & $p, q \geq 1$ \\ \hline 
& $\su(n, n)$ & $\sl(n, \C) \oplus \R$ & $n \geq 1$ \\ \hline 
$\star$ & $\so(n+2, \C)$ & $\so(n, 2)$ & $n \geq 3$ \\ \hline 
$\star$ & $\so(2n, \C)$ & $\so^\ast(2n)$ & $n \geq 3$ \\ \hline 
& $\so(p+1, q+1)$ & $\so(p, 1) \oplus \so(1, q)$ & $p, q \geq 0, \ (p,q) \neq (0,0), (1,1)$ \\ \hline 
& $\so(n, n)$ & $\so(n, \C)$ & $n \geq 3$ \\ \hline 
& $\so^\ast(4n)$ & $\sl(n, \H) \oplus \R$ & $n \geq 2$ \\ \hline 
& $\sp(n, \C)$ & $\sp(n, \R)$ & $n \geq 1$ \\ \hline 
& $\sp(n, \R)$ & $\sl(n, \R) \oplus \R$ & $n \geq 1$ \\ \hline 
& $\sp(n, n)$ & $\sp(n, \C)$ & $n \geq 1$ \\ \hline 
$\star$ & $\e_{6, \C}$ & $\e_{6(-14)}$ & --- \\ \hline 
$\star$ & $\e_{6(6)}$ & $\sp(2, 2)$ & --- \\ \hline 
& $\e_{6(-26)}$ & $\f_{4(-20)}$ & --- \\ \hline 
$\star$ & $\e_{7, \C}$ & $\e_{7(-25)}$ & --- \\ \hline
& $\e_{7(7)}$ & $\sl(4, \H)$ & --- \\ \hline 
& $\e_{7(-25)}$ & $\e_{6(-26)} \oplus \R$ & --- \\ 
\bottomrule \\
\end{tabular}
\caption{Irreducible symmetric pairs to which Corollary~\ref{cor:hyperbolic} is applicable} 
\label{table:hyperbolic}
\end{center}
\end{table}

\subsection{Examples obtained by direct computations}
\label{subsec:ex-direct}

In Subsections~\ref{subsec:ex-C-R} and \ref{subsec:ex-hyperbolic}, 
we systematically constructed examples of a homogeneous space 
that does not admit a compact Clifford--Klein form 
using Proposition~\ref{prop:simplification-ep}. 
In this subsection, we shall give some examples via direct verification of 
Proposition~\ref{prop:simplification-inv-poly}. 

To fix notations, we give explicit generators of $(S^k \g_\C^\ast)^{\g_\C}$ 
for $\g_\C = \sl(n, \C)$ and $\so(n, \C)$. 
Define invariant polynomials 
$f_k \in (S^k (\gl(n, \C)^\ast))^{\gl(n, \C)}$ $(k=1, 2, \dots, n)$ 
on the Lie algebra $\gl(n, \C)$ by 
\[
\det(\lambda I_n - X) = 
\lambda^n + f_1(X) \lambda^{n-1} + f_2(X) \lambda^{n-2} + \dots + f_n(X)
\quad (X \in \gl(n, \C)). 
\]
For the convenience, we put $f_0 = 1$ and $f_k = 0$ for $k<0$ and $k>n$. 
We use the same notation $f_k$ 
for the restriction of $f_k$ to $\sl(n, \C)$ or to $\so(n, \C)$. Then, 
\begin{fact}[{\cite[Ch.~VIII, \S 13]{BouLie7-9}}] 
\begin{enumerate}
\item[\textup{(1)}] 
The $n-1$ elements $f_2, f_3, \dots, f_n$ generate the algebra 
$(S (\sl(n, \C)^\ast))^{\sl(n, \C)}$ and are algebraically independent. 
We have $f_1 = 0$. 
\item[\textup{(2)}] If $n = 2m+1$, 
the $m$ elements $f_2, f_4, \dots, f_{2m}$ generate the algebra 
$(S (\so(n, \C)^\ast))^{\so(n, \C)}$ and are algebraically independent. 
We have $f_1 = f_3 = \dots = f_{2m+1} = 0$. 
\item[\textup{(3)}] If $n = 2m$, 
the $m$ elements $f_2, f_4, \dots, f_{2m-2}, \tilde{f}$ generate the algebra 
$(S (\so(n, \C)^\ast))^{\so(n, \C)}$ and are algebraically independent, where 
$\tilde{f} \in (S^m (\so(n, \C)^\ast))^{\so(n, \C)}$ 
is the Pfaffian of $n \times n$ skew-symmeric matrices. 
We have $f_1 = f_3 = \dots = f_{2m-1} = 0$ and $f_{2m} = \tilde{f}^2$. 
\end{enumerate}
\end{fact}
\begin{cor}\label{cor:so}
A homogeneous space 
\[
\SO_0(p+r, q)/\SO_0(p,q) \qquad (p, q, r \geq 1, \ q : \text{odd}) 
\]
does not admit a compact Clifford--Klein form. 
\end{cor}
\begin{proof}
Put $G = \SO_0(p+r, q)$ and $H = \SO_0(p,q)$. 
When $p$ is odd, this has been already proved (\cite[Cor.~1.6~(4)]{Mor15}). 
Let $p$ be even. Take a compact subgroup $C$ of $G$ to be 
\[
C = \SO(p+1) \times \SO(q) \subset \SO_0(p+r, q) = G 
\]
and define a homomorphism 
\begin{align*}
\phi : (S \h_\C^\ast)^{\h_\C} &= (S (\so(p+q, \C)^\ast))^{\so(p+q, \C)} \\
&\to (S (\so(p+1, \C)^\ast))^{\so(p+1, \C)} \otimes (S (\so(q, \C)^\ast))^{\so(q, \C)} = (S \c_\C^\ast)^{\c_\C}
\end{align*}
by 
\[
\phi(f_{2k}) = \sum_{i+j=k} f_{2i} \otimes f_{2j} 
\qquad (1 \leq k \leq \frac{p+q-1}{2}). 
\]
Then $C$ and $\phi$ satisfy the assumptions of 
Proposition~\ref{prop:simplification-inv-poly}. 
\end{proof}
\begin{cor}\label{cor:so-sym}
An irreducible symmetric space 
\[
\SO_0(p+r, q)/(\SO_0(p,q) \times \SO(r)) \qquad (p, q, r \geq 1, \ q : \text{odd}) 
\]
does not admit a compact Clifford--Klein form. 
\end{cor}
\begin{proof}
This is immediate from Corollary~\ref{cor:so} since $\SO(r)$ is compact. 
\end{proof}
\begin{cor}\label{cor:sl}
A homogeneous space 
\[
\SL(n, \R) / \SL(m, \R) \qquad (n > m \geq 2, \ m : \text{even})
\]
does not admit a compact Clifford--Klein form. 
\end{cor}
\begin{proof}
Put $G = \SL(n, \R)$ and $H = \SL(m, \R)$. 
Take a compact subgroup $C$ of $G$ to be 
\[
C = \SO(m+1) \subset \SL(n, \R) = G
\]
and define a homomorphism 
\[
\phi : (S \h_\C^\ast)^{\h_\C} = (S (\sl(m, \C)^\ast))^{\sl(m, \C)} 
\to (S (\so(m+1, \C)^\ast))^{\so(m+1, \C)} = (S \c_\C^\ast)^{\c_\C} 
\]
by $\phi(f_k) = f_k \ (2 \leq f_k \leq m)$. 
Then $C$ and $\phi$ satisfy the assumptions of 
Proposition~\ref{prop:simplification-inv-poly}. 
\end{proof}

\subsection{Enlargement of Lie groups}

The following lemma provides some other examples of a homogeneous space 
without compact Clifford--Klein forms. 
\begin{lem}\label{lem:enlarge}
Let $G/H$ be a homogeneous space of reductive type 
satisfying the assumptions of Proposition~\ref{prop:simplification-inv-poly}. 
Let $\tilde{G}$ be a Lie group containing $G$ as a closed subgroup. 
Let $L$ be a closed subgroup of $\tilde{G}$ such that 
$G \cap L = \{ 1 \}$, $L \subset Z_{\tilde{G}}(G)$ and 
$\tilde{G}/(H \times L)$ is a homogeneous space of reductive type. 
Then $\tilde{G}/(H \times L)$ satisfies the assumptions of 
Proposition~\ref{prop:simplification-inv-poly} 
(and therefore does not admit a compact Clifford--Klein form). 
\end{lem}
\begin{proof}
Let $C$ be a compact subgroup of $G$ and 
$\phi : (S \h_\C^\ast)^{\h_\C} \to (S \c_\C^\ast)^{\c_\C}$ 
a homomorphism of graded algebras satisfying the conditions (i)--(iii) of 
Proposition~\ref{prop:simplification-inv-poly}. 
Let $K_L$ be a maximal compact subgroup of $L$ and 
\[
\rest : (S \mfl_\C^\ast)^{\mfl_\C} \to (S (\k_L)_\C^\ast)^{(\k_L)_\C}
\]
denote the restriction map. 
Let $\tilde{C} = C \times K_L$ and 
\begin{align*}
\tilde{\phi} : (S (\h \oplus \mfl)_\C^\ast)^{(\h \oplus \mfl)_\C} 
&= (S \h_\C^\ast)^{\h_\C} \otimes (S \mfl_\C^\ast)^{\mfl_\C} \\ 
&\xrightarrow{\phi \otimes \rest} 
(S \c_\C^\ast)^{\c_\C} \otimes (S (\k_L)_\C^\ast)^{(\k_L)_\C} 
= (S (\c \oplus \k_L)_\C^\ast)^{(\c \oplus \k_L)_\C}. 
\end{align*}
Then $\tilde{C}$ and $\tilde{\phi}$ satisfy the conditions (i)--(iii) 
with respect to $\tilde{G}/(H \times L)$. 
\end{proof}
For instance, 
\begin{align*}
\SL(p_1 + \dots + p_n + q, \R) &/ (\SL(p_1, \R) \times \dots \times \SL(p_n, \R)) \\ 
&(n \geq 1, \ p_1, \dots, p_n \geq 2, \ \prod_i p_i : \text{even}, \ q \geq 1) 
\end{align*}
does not admit a compact Clifford--Klein form by 
Lemma~\ref{lem:enlarge} and the proof of Corollary~\ref{cor:sl}. 

\subsection{Relation with an earlier result}

We proved in \cite{Mor15+} the following result: 

\begin{fact}[{\cite[Th.~1.2~(2)]{Mor15+}}]\label{fact:earlier}
Let $G$ be a Lie group and 
$H$ its closed subgroup with finitely many connected components. 
Let $K_H$ be a maximal compact subgroup of $H$. If the homomorphism
\[
i: H^N(\g, \h; \R) \to H^N(\g, \k_H; \R) \qquad (N = \dim G - \dim H)
\]
is not injective, there is no compact manifold locally modelled on $G/H$. 
\end{fact}

If the homomorphism 
$i : H^N(\g, \h; \R) \to H^N(\g, \k_H; \R)$ is not injective, 
\[
\im ( i : H^N(\g, \h; \R) \to H^N(\g, \t_H; \R) ) \subset I^N
\]
trivially holds. 
Thus Theorem~\ref{thm:main} yields the following corollary, 
which is slightly weaker than Fact~\ref{fact:earlier}: 

\begin{cor}
Let $G$ be a connected linear Lie group and $H$ its connected closed subgroup. 
Let $K_H$ be a maximal compact subgroup of $H$. If the homomorphism 
\[
i: H^N(\g, \h; \R) \to H^N(\g, \k_H; \R) \qquad (N = \dim G - \dim H)
\]
is not injective, $G/H$ does not admit a compact Clifford--Klein form. 
\end{cor}

\begin{rem}
Recall from \cite[\S 4]{Mor15} that 
$i : H^N(\g, \h; \R) \to H^N(\g, \t_H; \R)$ is injective if and only if so is 
$i : H^N(\g, \h; \R) \to H^N(\g, \k_H; \R)$, for 
$i : H^\bl(\g, \k_H; \R) \to H^\bl(\g, \t_H; \R)$ 
is always injective by a variant of the splitting principle 
(\cite[Th.~6.8.2]{Gui-Ste99}). 
\end{rem}

\subsection*{Acknowledgements}

The author wishes to express his sincere gratitude to 
Professor Toshiyuki Kobayashi for his advice and encouragement. 
This work was supported by JSPS KAKENHI Grant Number 14J08233 
and the Program for Leading Graduate Schools, MEXT, Japan.

\end{document}